\documentclass[11pt]{amsart}
\usepackage[utf8]{inputenc}
\usepackage{mathptmx} 
\usepackage[scaled=0.90]{helvet}
\usepackage{courier} 
\usepackage[toc,page]{appendix}
\usepackage[T1]{fontenc}

\usepackage{float}
\usepackage[margin=1.2in]{geometry}
\usepackage[pdftex,bookmarks,colorlinks,breaklinks]{hyperref} 
\usepackage{amsmath}
\usepackage{amsthm}
\usepackage{amsfonts}
\usepackage{amssymb}
\usepackage{graphicx}
\usepackage{amssymb,amsmath,amsthm}
\usepackage{mathrsfs}
\usepackage{graphicx}

\usepackage{color}

\newtheorem{theorem}{Theorem}

\newtheorem{definition}{Definition}

\newtheorem{remark}{Remark}

\newcommand{\hreff}[1]{\hyperref[#1]{\ref{#1}}}

\newcommand{\av}[2]{\langle {#1}\rangle_{{}_{#2}}}

\newcommand{\df}{\buildrel\mathrm{def}\over=}
\newcommand{\s}{\mathbf}
\newcommand{\eps}{\varepsilon}
\makeindex

\begin{document}
\title{Bellman function approach to the sharp constants in uniform convexity}

\author{Paata Ivanisvili}
\address{Department of Mathematics, Michigan State University, East Lansing, MI 48823, USA}
\email{ivanisvi@math.msu.edu}
\urladdr{http://math.msu.edu/~ivanisvi}

\begin{abstract}
We illustrate Bellman function technique in finding the modulus of uniform convexity of $L^{p}$ spaces. 
\end{abstract}

\makeatletter
\@namedef{subjclassname@2010}{
  \textup{2010} Mathematics Subject Classification}
\makeatother

\subjclass[2010]{42B20, 42B35, 47A30}

\keywords{Uniform convexity, Bellman function, Concave envelopes}

\maketitle

\setcounter{equation}{0}
\setcounter{theorem}{0}

\section{Uniform Convexity} 

Let $I$ be an interval of the real line. For an integrable function $f$ over $I$, we set $\langle f \rangle_{I} \df  \frac{1}{|I|}\int_{I} f(s)ds$, and $\|f\|_{p} \df \langle |f|^{p}\rangle_{I}^{1/p} $. 
 We recall the definition of uniform convexity of a normed space $(X,\|\cdot\|)$ (see~\cite{Cl}). 
\begin{definition}
\textup{(}Clarkson~'36\textup{)} $X$ is uniformly convex if $\; \forall \eps >0$, $\exists \delta>0$ s.t. if $\|x\|=\|y\|=1$ and $\|x-y\|\geq \eps$, then $\left\|\frac{x+y}{2}\right\|\leq 1-\delta$.
\end{definition}
{\em Modulus of convexity}  of the normed space $X$ is defined as follows:
\begin{align*}
&\delta_{X}(\eps) = \inf \left\{  \left(1- \frac{\|f+g\|}{2} \right) \; : \; \|f\|=1,\; \|g\|=1,\; \|f-g\| \geq \eps \right\}.
\end{align*}
\begin{remark}
$(X,\|\cdot \|)$ space is {\em uniformly convex} iff $\delta_{X}(\eps) >0$. 
\end{remark}
O.~Hanner (see~\cite{Ha}) gave an elegant proof of finding the constant $\delta_{L^{p}}(\eps)$ in $L^{p}([0,1])$ space for $p\in  (1,\infty)$ in 1955. He proved two necessary inequalities (further called Hanner's inequalities) in order to obtain constant $\delta_{L^{p}}(\eps)$. Namely, 
\begin{align}\label{han}
\|f+g\|_{p}^{p}+\|f-g\|_{p}^{p}\geq (\|f\|_{p}+\|g\|_{p})^{p}+|\|f\|_{p}-\|g\|_{p}|^{p},  \quad   p \in [1,2],
\end{align}
and the inequality (\ref{han}) is reversed if $p\geq 2$. 
 Hanner mentions in his note~\cite{Ha} that his  proof is a reconstruction of some Beurling's ideas  given at a seminar in Upsala in 1945. In \cite{BKL} non-commutative case of Hanner's inequalities was investigated. Namely, Hanner's inequality holds for $p \in [1,3/4]\cup [4,\infty)$, and the case $p \in (3/4,4)$ (where $p\neq 2$) was left open. 
 
  In this note we present ``general'' systematic approach  in finding the constant $\delta_{L^{p}}(\eps)$ by Bellman function technique, where absolutely no background is required, only elementary calculus. We also show that the Bellman function (\ref{bellman}), which arise naturally, is a minimal concave function with the given boundary condition (\ref{bc}).

\section{Minimal concave functions over the obstacle}
Let  $\Omega \subset  \mathbb{R}^{n}, m:\Omega \to \mathbb{R}^{k}$ and $H :\Omega \to \mathbb{R}$. Let $\Omega(I)$ denotes the class of piecewise constant  vector-valued functions 
$\varphi :I \to \Omega $, and let $\mathrm{conv}(\Omega)$ be the convex hull of the set $\Omega$.
We define the Bellman function as follows 
\begin{align*}
B(x) =\sup_{\varphi \in \Omega(I)} \{ \av{H(\varphi)}{I}\,: \quad \av{m(\varphi)}{I}=x\}.
\end{align*}
\begin{theorem}\label{mainth}
The following properties hold:
\begin{itemize}
\item[1.] $B$ is defined on the convex set $\mathrm{conv}[m(\Omega)]$;
\item[2.] $B(m(y))\geq H(y)$ for all $y \in \Omega$;
\item[3.] $B$ is concave function;
\item[4.] $B$ is minimal among those who satisfy properties 1,2 and 3. 
\end{itemize}
\end{theorem}
\begin{proof}
Fist we show the property 1. Let $\mathrm{Dom}\, B$ denotes the domain where $B$ is defined. Since $m(\varphi) \in \mathrm{conv}[m(\Omega)]$ we have $\av{m(\varphi)}{I} \in \mathrm{conv}[m(\Omega)]$. Therefore $\mathrm{Dom}\, B \subseteq \mathrm{conv}[m(\Omega)]$. Now we show the opposite inclusion. 
Carath\'eodory's theorem implies that for any $x \in \mathrm{conv}[m(\Omega)]$ we have $x = \sum_{j=1}^{n+1}a_{j}x_{j}$, where $a_{j} \geq 0,$ $\sum_{j=1}^{n+1}a_{j}=1$ and $x_{j} \in m(\Omega)$. Let the points $y_{j}$ be such that $m(y_{j})=x_{j}$. We choose $\varphi$ so that $|\{t \in I \, :\; \varphi(t)=y_{j}\}|=a_{j}|I|$. Then $\varphi \in \Omega(I)$. Hence, 
\begin{align*}
\av{m(\varphi)}{I}=\frac{1}{|I|}\int_{I}m(\varphi(t))dt = \sum_{j=1}^{n+1}\frac{1}{|I|}\int_{\{t : \varphi(t)=y_{j}\}} m(\varphi(t))dt = \sum_{j=1}^{n+1}\frac{1}{|I|}x_{j} |I|a_{j}=x
\end{align*}

Now we show the property 2. Let $\varphi_{0}(t)=y, \, t \in I$. Then $\av{m(\varphi_{0})}{I}=m(y)$. Thus 
\begin{align*}
B(m(y))=\sup_{\varphi \in \Omega(I) : \av{m(\varphi)}{I}=m(y)} \av{H(\varphi)}{I}\geq \av{H(\varphi_{0})}{I}=H(y).
\end{align*}

Now we show the property 3. It is enough to show that $B(\theta x + (1-\theta) y) \geq \theta B(x) + (1-\theta)B(y)$ for all $x,  y \in \mathrm{conv}[m(\Omega)]$ and $\theta \in [0,1]$. There exist functions $\varphi, \psi \in \Omega(I)$ such that $\av{m(\varphi)}{I}=x, \av{m(\psi)}{I}=y$ and 
\begin{align*}
\av{H(\varphi)}{I}\geq B(x)-\eps, \quad \av{H(\psi)}{I} > B(y)-\eps.
\end{align*}
We split interval  $I$ by two disjoint subintervals $I_{1}$ and $I_{2}$  so that $|I_{1}|=\theta |I|$. Let $L_{j} : I_{j} \to I$  be a linear bijections. We consider the concatenation as follows 
\begin{align*}
\eta(t) = 
\begin{cases}
\varphi(L_{1}(t)), \quad t \in I_{1},\\
\varphi(L_{2}(t)), \quad t \in I_{2}. 
\end{cases}
\end{align*}
Clearly $\eta(t) \in \Omega(I)$, $\av{m(\eta)}{I}=\theta x + (1-\theta)y$ and 
\begin{align*}
B(\theta x + (1-\theta)y) \geq \av{H(\eta)}{I}=\theta \av{H(\varphi)}{I}+ (1-\theta)\av{H(\psi)}{I} > \theta B(x)+(1-\theta)B(y)-\eps.
\end{align*}

Now we show the property 4. Let $G$ satisfies properties 1,2 and 3. Then Jensen's inequality implies that  for any $\varphi \in \Omega(I)$ we have 
\begin{align*}
\av{H(\varphi)}{I}\leq \av{G(m(\varphi))}{I}\leq G(\av{m(\varphi)}{I})=G(x)
\end{align*}
\end{proof}
We make the following simple observations:
\begin{itemize}
\item[1.] Under some mild assumptions on $H$ and $m$ one can obtain Theorem \ref{mainth} for the measurable class of functions $\varphi :I \to \Omega$. For example, if one defines $\Omega^{*}(I)$ to be the class of measurable functions $\varphi : I \to \Omega$ so that $m(\varphi)$ is integrable, and $H(\varphi)$ is measurable, then we can introduce the Bellman function as follows 
\begin{align*}
B^{*}(x) =\sup_{\varphi \in \Omega^{*}(I)} \{ \av{H(\varphi)}{I}\,: \quad \av{m(\varphi)}{I}=x\}.
\end{align*}
Then it is clear that $B^{*} \geq B$. In order to show the opposite inequality $B^{*}\leq B$ one has to justify Lebesgue's dominated convergence theorem (or Fatou's lemma). To do this it is enough to require that $H$ and $m$ are continuous maps, $B$ is not identically infinity, and $m(\varphi_{n})$ has an integrable majorant whenever $\varphi_{n} \to \varphi$ a.e.  for some piecewise constant functions $\varphi_{n}$.

\item[2.] We note that the property 2  of Theorem \ref{mainth}, namely, $B(m(y))\geq H(y)$ for all $y \in \Omega$, can be rewritten as follows $B(x) \geq \sup_{y \in m^{-1}(x)}H(y)$ for all $y\in m(\Omega)$. We define the obstacle as follows $R(x) \df \sup_{y \in m^{-1}(x)} H(y)$. Then it is clear that 
\begin{align*}
B(x) = \sup_{\varphi \in \Omega(I)}\{ \av{R(m(\varphi))}{I}\, : \quad \av{m(\varphi)}{I}=x\},
\end{align*}
and the property 2 takes the form $B(x) \geq R(x)$ for all $x \in m(\Omega)$. 

\item[3.]The Bellman function $B$ does not depend on the choice of the interval $I$, and $B(x)=R(x)$ at the extreme points of the set $\mathrm{conv}[m(\Omega)]$. 

\item[4.] If $\mathrm{conv}[m(\Omega)]=m(\Omega)$ and the obstacle $R(x)$ is concave then $B(x)=R(x)$ for all $x \in m(\Omega)$.
\end{itemize}
Now we show the application of Theorem~\ref{mainth} on the example of uniform convexity of $L^{p}$ spaces. 
\section{Bellman function in uniform convexity}
\subsection{Domain and the boundary condition}
 The definition of the modulus of uniform convexity tells us to consider the following function
\begin{align}\label{bellman}
B(x) = \sup_{f,g} \{ \av{|\theta f+(1-\theta)g|^{p}}{I}, \; \av{(|f|^{p},|g|^{p},|f-g|^{p})}{I}=x \}.
\end{align}
where $\theta=1/2$. Then it is clear that $\delta_{L^{p}}(\eps) = 1-\sup_{2^{p}\geq x_{3} \geq \eps^{p}} (B(1,1,x_{3}))^{1/p}$.
Theorem~\ref{mainth} (together with the observation 1) implies that if we set $\Omega=\mathbb{R}^{2}$, $\varphi=(f,g)$ $m(x,y)=(|x|^{p},|y|^{p},|x-y|^{p})$ and  $H(x,y)=|\theta x+(1-\theta)y|^{p}$, then $B$ is the minimal concave function on the domain $\mathrm{conv}[m(\Omega)]$ such that $B(m(x,y)) \geq H(x,y)$.

If we set $\av{m(\varphi)}{I}=x=(x_{1},x_{2},x_{3})$ then note that all variables $x_{1},x_{2},x_{3}$ are nonnegative. Note also that
\begin{align*}
&\Lambda=\mathrm{conv}[m(\Omega)] = \{ x_{1},x_{2},x_{3} \geq 0,\\
&x_{1}^{1/p}+x_{2}^{1/p}\geq x_{3}^{1/p},\; x_{2}^{1/p}+x_{3}^{1/p}\geq x_{1}^{1/p},\; x_{3}^{1/p}+x_{1}^{1/p}\geq x_{2}^{1/p}\}.
\end{align*}
$\Lambda$ is the convex cone and $\partial \Lambda=m(\Omega)$.
 Minkowski's inequality implies that whenever $x \in m(\Omega)$  we must have $f=\lambda g$ for an appropriate $\lambda$. This allows us to find boundary data for the function $B$. 
\begin{align} \label{bc}
B(x_{1},x_{2},x_{3}) = 
\begin{cases} 
|\theta x_{1}^{1/p}-(1-\theta)x_{2}^{1/p}|^{p}, &  x_{1}^{1/p}+x_{2}^{1/p}=x_{3}^{1/p},\\
(\theta x_{3}^{1/p}+x_{2}^{1/p})^{p},  & x_{2}^{1/p}+x_{3}^{1/p}=x_{1}^{1/p},\\
(x_{1}^{1/p}+(1-\theta)x_{3}^{1/p})^{p}, & x_{3}^{1/p}+x_{1}^{1/p}=x_{2}^{1/p},
\end{cases}
\end{align}
 where $x_{1},x_{2}$ and $x_{3}$ are nonnegative numbers. 
 Thus, Theorem~\ref{mainth} implies that $B$ is minimal concave function on $\Lambda$ with the given boundary condition (\ref{bc}).  
\subsection{Sharp constants in uniform convexity}
We recall that
\begin{align*}
\delta_{L^{p}}(\eps) = 1-\sup_{2^{p}\geq x_{3} \geq \eps^{p}} (B(1,1,x_{3}))^{1/p}.
\end{align*}
The reader can try to find the function $B(x_{1},x_{2},x_{3})$. However, one can avoid finding the exact value of $B(x_{1},x_{2},x_{3})$, and by Theorem~\ref{mainth} (namely, property 4) one can present an appropriate concave function $\mathcal{B}$ which majorize $B$ on $\partial \Lambda$ (and hence on $\Lambda$), however, gives the exact value of $\delta_{L^{p}}(\eps)$. 

We consider the case $p\geq 2$. Let's consider the following function
\begin{align}\label{bb1}
\mathcal{B}(x_{1},x_{2},x_{3})=\frac{x_{1}+x_{2}}{2}-\frac{x_{3}}{2^{p}}. 
\end{align}
Surely $\mathcal{B}$ is concave in $\Lambda$ and $(\mathcal{B}-B)|_{\partial \Omega} \geq 0$ (see  Appendieces).  Therefore, Theorem~\ref{mainth} implies that $\mathcal{B} \geq B$ in $\Lambda$.  Thus, 
\begin{align*}
&\delta(\eps) = 1-\sup_{2^{p}\geq x_{3} \geq \eps^{p}} (B(1,1,x_{3}))^{1/p}\geq 
1-\sup_{2^{p}\geq x_{3} \geq \eps^{p}} (\mathcal{B}(1,1,x_{3}))^{1/p}=
\\
&1-(\mathcal{B}(1,1,\eps^{p}))^{1/p}=1-\left(1-\frac{\eps^{p}}{2^{p}} \right)^{1/p}.
\end{align*}
If we show that $B(1,1,\eps^{p})= \mathcal{B}(1,1,\eps^{p})$ then this would  imply that the estimate obtained above for $\delta_{L^{p}}(\eps)$ is sharp.  Homogeneity of the functions $\mathcal{B}$ and $B$ (i.e. $B(\lambda \s{x})=\lambda B(x)$ for all $\lambda \geq 0$) implies that it is enough to prove the equality $B(\eps^{-p},\eps^{-p},1)= \mathcal{B}(\eps^{-p},\eps^{-p},1)$. We show that $B(s,s,1)= \mathcal{B}(s,s,1)$ for all $s \geq 2^{-p}$.
Take an arbitrary $s \in (2^{-p}, \infty)$,
 Consider the points $A = (2^{-p},2^{-p},1)$ and $D(s) =(s,|1-s^{1/p}|^{p},1)$. Clearly $A,D(s) \in \partial \Lambda$. Let $L_{s}(\s{x})$ be a linear function such that  $L_{s}(A)=B(A)=0$, $L_{s}(D(s)) = B(D(s))=(-1/2+s^{1/p})^{p}$. Concavity of  $B$ implies that $B \geq L$ on the chord $[A,D(s)]$ joining the points $A$ and $D(s)$. On the other hand one can easily see that for all $a>0$, there exists sufficiently large $s$ such that we have $|L_{s}(\s{x}) - \mathcal{B}(\s{x})|<a$ for the points $\s{x}$ belonging to the chord $[A,D(s)]$. Continuity of $B$ and $\mathcal{B}$ finishes the story. 
 
 Now we consider the case $1<p<2$. We set 
 \begin{align*}
 (g(s),f(s)) = (|1-s^{1/p}|^{p},(s^{1/p}-1/2)^{p}) \quad \text{for} \quad s \geq 2^{-p}.
 \end{align*}
  Let $s^{*} \in [2^{-p},\infty)$ be the solution of the equation $2\eps^{-p}=s^{*}+g(s^{*})$. Consider the following function 
 \begin{align}\label{bb2}
 \mathcal{B}(x_{1},x_{2},x_{3})=x_{3}f(s^{*})+\frac{f'(s^{*})}{1+g'(s^{*})}\left[x_{1}+x_{2}-2\eps^{-p}x_{3} \right]. 
 \end{align}
Surely $\mathcal{B}$ is concave in $\Lambda$ and $(\mathcal{B}-B)|_{\partial \Omega} \geq 0$ (see Appendieces).  Therefore, Theorem~\ref{mainth} implies that $\mathcal{B} \geq B$ in $\Lambda$. The inequality $B'_{x_{3}}(1,1,x_{3})\leq 0$ follows from the inequality $f(s)(1+g'(s))-f'(s)(s+g(s))\leq 0$ for all $s \geq 2^{-p}$,  which can be seen by direct computation.  Thus, 
\begin{align*}
&\delta(\eps) = 1-\sup_{2^{p}\geq x_{3} \geq \eps^{p}} (B(1,1,x_{3}))^{1/p}\geq 
1-\sup_{2^{p}\geq x_{3} \geq \eps^{p}} (\mathcal{B}(1,1,x_{3}))^{1/p}=
\\
&1-(\mathcal{B}(1,1,\eps^{p}))^{1/p}=1-\eps \left(f(s^{*})\right)^{1/p}.
\end{align*}

If we show that $B(1,1,\eps^{p})= \mathcal{B}(1,1,\eps^{p})$ then this would  imply that the estimate obtained above for $\delta_{L^{p}}(\eps)$ is sharp.  Homogeneity of the functions $\mathcal{B}$ and $B$ implies that it is enough to prove the equality $B(\eps^{-p},\eps^{-p},1)= \mathcal{B}(\eps^{-p},\eps^{-p},1)$. 

Consider the points $A = (s^{*},g(s^{*}),1)$ and $D =(g(s^{*}),s^{*},1)$. Clearly $A,D \in \partial \Lambda$. Let $L(\s{x})$ be a linear function such that  $L(A)=B(A)=f(s^{*})$, $L(D) = B(D)=f(s^{*})$. Concavity of  $B$ implies that $B(\s{x}) \geq L(\s{x})$ on the chord $[A,D]$ joining the points $A$ and $D$. On the other hand one can easily see that $\mathcal{B}(\s{x}) = L(\s{x})$ for the points $\s{x}$ belonging to the chord $[A,D]$. Thus the fact $(\eps^{-p},\eps^{-p},1) \in  [A,D]$, which follows from the equality $2\eps^{-p}=s^{*}+g(s^{*})$ finishes the proof.  It is clear that the constant $\delta_{L^{p}}(\eps) \in [0,1]$ is the solution of the equation  
\begin{align*}
\left( 1-\delta(\eps)+\frac{\eps}{2}\right)^{p} + \left|1-\delta(\eps)-\frac{\eps}{2} \right|^{p}=2.
\end{align*}

Thus we finish the current note. One can ask how did we find the functions $\mathcal{B}$. These functions are tangent planes to the graph of the actual Bellman function $B$ at point $(1,1,\eps^{p})$. Unlike the actual Bellman function $B$, which has the implicit expression, the tangent planes $\mathcal{B}$ to the graph $B$ have simple expression, so it was easy to work with tangent planes $\mathcal{B}$, rather than with actual Bellman function $B$.

At the end of the note I would like to mention that the Bellman function (\ref{bellman}) was born during the listening seminar given by N.~K.~Nikolskij at Chebyshev Laboratory in Saint Petersburg. Idea of investigating  such type of Bellman functions comes from A.~Volberg's seminars given at the same laboratory (see also \cite{iOSVZ}). 
I would like to thank P.~Zatitskiy and D.~Stolyarov for long discussion and construction of the actual Bellman function (\ref{bellman}). They should be considered as co-authors of this note. 

\section{Appendieces}
We show that $(\mathcal{B}-B)|_{\partial \Lambda} \geq 0$. Homogeneity of $B$ and $\mathcal{B}$ implies that without loss of generality we can assume $x_{3}=1$. In this case boundary condition (\ref{bc}) can be rewritten as follows $B(s,g(s),1)=B(g(s),s,1)=f(s)$ for all $s\geq 2^{-p}$.  It is enough to show that $U(s) = \mathcal{B}(s,g(s),1)-f(s)  \geq 0$ for $s \geq 2^{-p}$, because the other cases can be covered by symmetry i.e. $\mathcal{B}(s,g(s),1)=\mathcal{B}(g(s),s,1)$ and $B(s,g(s),1)=B(g(s),s,1).$

{\bf Case  $p\geq 2$.}  Let $V(s)=U(s^{p})$. Note that  $V(2^{-1})=0$, and $V'(s) \geq 0$ for $s \geq 1$. Note that $W(s)=V'(s^{-1})p^{-1}2 s^{p-1}=1-(s-1)^{p-1}-2(1-s/2)^{p-1}$ is a concave function for $s \in (1,2)$, and it is nonnegative at the endpoints of this interval, hence $V'(s) \geq 0$ for $s \geq 2^{-p}$. 

{\bf Case $1<p<2$.} Note that $U(s^{*})=0$,  the function $\frac{f'(s)}{1+g'(s)}$ is decreasing and $1+g'(s) > 0$ for $s \geq 2^{-p}$. It is clear that $U'(s) (1+g'(s)) = \frac{f'(s^{*})}{1+g'(s^{*})}-\frac{f'(s)}{1+g'(s)}$. Thus we see that $U'(s) < 0$ for $s \in [2^{-p},s^{*})$, $U'(s^{*})=0$, and $U'(s) >0$ for $s \in (s^{*},\infty)$. Hence $U(s) \geq 0$.


\begin{thebibliography}{99999}



\bibitem{Ha} O.~Hanner, \emph{On the uniform convexity of $L^{p}$ and $\ell^{p}$},	Ark. Mat. 3 (1956), 239-244. MR0077087(17:987d)

\bibitem{Cl} J.~A.~Clarkson, \emph{Uniformly convex spaces}, Trans. Amer. Math. Soc 40 (1936), no. 3, 396-414. MR1501880

\bibitem{BKL} K.~Ball, E.~A.~Carlen, E.~H.~Lieb {\em Sharp uniform convexity and smoothness inequalities for trace norms.} Invent. Math.  115(3), pp.463--482 (1994)

\bibitem{iOSVZ} P.~Ivanishvili, N.~Osipov, D.~Stolyarov, V.~Vasyunin, P.~Zatitskiy,
{\em Bellman function for extremal problems in BMO,} to appear in Transactions of the Amer. Math. Soc.
\end{thebibliography}
\end{document}